\documentclass{amsart}
\usepackage{amssymb}
\usepackage{amsmath}
\usepackage{amsfonts}
\usepackage{hyperref}
\usepackage{yhmath}
\usepackage{framed}
\newtheorem{theorem}{Theorem}[section]
\newtheorem{lemma}[theorem]{Lemma}

\newtheorem{remark}[theorem]{Remark}

\newcommand{\R}{\mathbb{R}}

\begin{document}
\title{Ill-posedness results for generalized Boussinesq equations}

\author{Dan-Andrei Geba, A. Alexandrou Himonas, and David Karapetyan}

\address{Department of Mathematics, University of Rochester, Rochester, NY 14627}
\address{Department of Mathematics, University of Notre Dame, Notre Dame, IN 46556}
\address{Department of Mathematics, University of Rochester, Rochester, NY 14627}
\date{}

\begin{abstract}
In this article, we investigate both the periodic and non-periodic initial value problems for generalized Boussinesq equations. We show that the associated flow map is not smooth for a range of Sobolev indices, thus providing a threshold for the regularity needed to perform a Picard iteration for these problems. 
\end{abstract}


\maketitle

\section{Introduction}

In this paper, we are concerned with the Cauchy problem for the generalized Boussinesq equation
\begin{equation}
\left\{
\begin{array}{l}
u_{tt}-u_{xx}+u_{xxxx}+(f(u))_{xx}\,=\,0, \qquad u=u(t,x): \mathbb{R}_+\times M \to \mathbb{R},\\
\\
u(0,x)\,=\,u_0(x),\qquad u_t(0,x)\,=\,u_1(x),\\
\end{array}\right.
\label{main}
\end{equation}
where $M=\mathbb{R}$ (the non-periodic case) or $M=\mathbb{T}$ (the periodic case). Such an  equation, with $f(u)=4u^3-6u^5$, was derived by Falk, Laedke, and Spatschek \cite{FLS} in the study of shape-memory alloys. For $f(u) =  u^{2}$ one recovers the classical ``good'' Boussinesq equation, which is known to describe electromagnetic waves in nonlinear dielectrics \cite{T93}. 

As the ``good'' equation has been the subject of quite a few recent investigations (e.g., \cite{KT10}, \cite{OS12}, \cite{K12}), where almost-complete references for this problem have been discussed, we will mention here only previous results relevant to generalized type equations. 

In the non-periodic case, Bona and Sachs \cite{BS} proved local well-posedness (LWP) of \eqref{main} for $f\in C^\infty(\mathbb{R})$, $f(0)=0$, and $(u_0,u_1)\in H^s \times H^{s-2}$, $s>5/2$. Moreover, for pure power nonlinearities, $f(u)\simeq\pm |u|^{p-1}\,u$, with $1<p<5$, they showed nonlinear stability of solitary wave solutions and found sufficient conditions for the global existence of smooth solutions.  The LWP was then improved for pure power nonlinearities by Tsutsumi and Matahashi \cite{TM}, who demonstrated that this holds for $u_0\in H^1$ and $u_1=\phi_{xx}$, with $\phi\in H^1$. This was followed by Linares \cite{L93}, who proved LWP for $(u_0,u_1)=(g, h_x)$ when either $(g,h) \in H^1\times L^2$ and $p>1$ or $(g,h) \in L^2 \times \dot{H}^{-1}$  and $1<p\leq 5$. Furthermore, one has global well-posedness (GWP) in the former setting if $\|g\|_{H^1}+\|h\|_{L^2}$ is sufficiently small. Finally, Farah \cite{F092} showed LWP for $u_0\in H^s$ and $u_1=\phi_{xx}$, with $\phi\in H^s$, when
\[
p>1 \quad \text{and} \quad s\geq \max \left\{0,\,\frac{1}{2}-\frac{2}{p-1}\right\}. \]
For the defocusing problem $f(u)= -|u|^{p-1}\,u$, where $p\geq 3$ is an odd integer, Farah and Linares \cite{FL} and Farah and Wang \cite{FW12}, respectively, proved GWP for 
\[
u(0)\in H^s, \quad u_t(0)=h_x , \quad h \in H^{s-1}, \quad s> 1-\frac{2}{3(p-1)}.
\]

For the periodic problem associated to \eqref{main} there are considerably fewer results. In fact, to our knowledge, the only LWP result is due to Grillakis and Fang \cite{FG96}, who proved this for $(u_0,u_1)\in H^s \times H^{s-2}$  with 
\[
s\geq \max \left\{0,\,\frac{1}{2}-\frac{1}{p-1}\right\}. \]
They also showed that, for
\[
f(u)\,=\,\lambda |u|^{q-1}u - |u|^{r-1}u, \quad 1<q<r, \quad \lambda\in \mathbb{R},\] 
GWP holds if $(u_0,u_1)\in H^1\times H^{-1}$. 

An important fact revealed by this literature review is the absence of ill-posedness (IP) results for generalized Boussinesq equations. This  was also mentioned in  \cite{FW12}, where it was posed as an interesting open problem. The goal of this article is to answer this question, our results addressing both the non-periodic and periodic cases. 

\begin{theorem}
Consider the Cauchy problem \eqref{main} with  $f(u)=\pm\, u^p$, where $p>1$ is an integer, and let 
\begin{equation}
s\,<\,\left\{
\begin{array}{l}
-\frac{2}{p}, \quad \text{for} \quad p \ \text{odd},\\
\\
-\frac{1}{p}, \quad \text{for} \quad p \ \text{even}.\\
\end{array}\right.
\label{sp}
\end{equation} 
Then there exists $T>0$ such that the flow map
\begin{equation}
S(t):  H^{s} \times H^{s-2}(M) \to H^{s}(M), \quad S(t)(u_0,u_1)\,=\,u(t), \quad 0< t < T,
\label{solmap}
\end{equation}
for \eqref{main} is not $C^p$ Fr\'{e}chet-differentiable at zero.
\label{mainth}
\end{theorem}

\begin{remark}
A careful inspection of the proof for Theorem \ref{mainth} reveals that in the particular case of the ``good'' Boussinesq equation, i.e., $p=2$, we obtain a much stronger result which was previously obtained by Kishimoto \cite{K12}. Our argument, which one needs to combine with a rather general method for proving IP introduced by Bejenaru and Tao \cite{BT06}, shows that the flow map is in fact not continuous at zero for $s<-\frac 12$.
\end{remark}

\begin{remark}
Our results apply also to generalized Boussinesq equations with
\[
u: \mathbb{R}_+\times M \to \mathbb{C},
\]
as all of the functions involved in our argument take real values. In these cases, the nonlinearity has the profile
\[
f(u)\,=\,\pm\,u^{p-q}\,\overline{u}^q,
\]
where both $p>1$ and $q\geq 0$ are integers and $p\geq q$. 
\end{remark}

\section{Preliminaries}
Using Duhamel's principle, we can rewrite \eqref{main} in the integral form
\begin{equation}
S(t)(u_0,u_1)\,=\,L(u_0,u_1)(t)\,\pm\, \int_0^t\,L \left( 0,\left((S(\tau)(u_0,u_1))^p\right)_{xx}\right)(t-\tau)\,d\tau,
\end{equation}
where
\begin{equation}
\widehat{L(u_0,u_1)}(t,\xi)\,=\,\cos(t \lambda(\xi))\, \widehat{u}_0(\xi)+\frac{\sin(t \lambda(\xi))}{\lambda(\xi)} \,\widehat{u}_1(\xi), \quad \lambda(\xi)\,=\,\sqrt{\xi^2+\xi^4}.
\label{L}
\end{equation}

If we assume that the flow map is smooth, its Fr\'{e}chet derivative at $(u_0,u_1)$ evaluated for $(v_0,v_1)$ is given by
\begin{equation}
\aligned
&DS(t)_{(v_0,v_1)}(u_0,u_1)\\
=L(v_0,v_1)(t)
&\pm p\int_0^t\,L \left( 0,(DS(\tau)_{(v_0,v_1)}(u_0,u_1)\left(S(\tau)(u_0,u_1)\right)^{p-1})_{xx}\right)(t-\tau)\,d\tau,
\endaligned\label{ds}
\end{equation}
which, combined with $S(t)(0,0)=0$ due to the well-posedness of the problem, yields
\begin{equation}
DS(t)_{(v_0,v_1)}(0,0)\,=\,L(v_0,v_1)(t).
\label{ds0}
\end{equation} 

Computations of further Fr\'{e}chet derivatives lead to 
\begin{equation}
\aligned
 D^pS(t)&_{\left((v^1_0,v^1_1),\ldots, (v^p_0,v^p_1)\right)}(0,0)\\
=\,c_p\,\int_0^t\,L& \left( 0,\left(DS(\tau)_{(v^1_0,v^1_1)}(0,0)\,\cdot \ldots\cdot\,DS(\tau)_{(v^p_0,v^p_1)}(0,0)\right)_{xx}\right)(t-\tau)\,d\tau\\
=\,c_p\,\int_0^t\,L& \left( 0,\left(L(v^1_0,v^1_1)(\tau)\,\cdot \ldots\cdot\,L(v^p_0,v^p_1)(\tau)\right)_{xx}\right)(t-\tau)\,d\tau,
\endaligned
\label{dps}
\end{equation}
where $c_p$ is a constant strictly depending on $p$. If the flow map would be $C^p$ at zero, 
then 
\begin{equation}
 \left\|D^pS(t)_{\left((v^1_0,v^1_1),\ldots, (v^p_0,v^p_1)\right)}(0,0)\right\|_{H^s}\,\lesssim\,\prod_{j=1}^{p}\, \left(\left\|v^j_0\right\|_{H^s}+\left\|v^j_1\right\|_{H^{s-2}}\right)
\end{equation}
would be true, at least for $(v^1_0,v^1_1),\ldots, (v^p_0,v^p_1)$ all lying within a sufficiently small ball of $ H^s\times H^{s-2}$ centered at the origin. In the particular case when $(v^1_0,v^1_1)=\ldots= (v^p_0,v^p_1)=(u_0,u_1)$, this estimate reads as
\begin{equation}
 \left\|D^pS(t)_{\left((u_0,u_1),\ldots, (u_0,u_1)\right)}(0,0)\right\|_{H^s}\,\lesssim\,\left(\|u_0\|_{H^s}+\|u_1\|_{H^{s-2}}\right)^p.
\end{equation}

Using the notation
\[
A_p(u_0,u_1)(t)\,=\,D^pS(t)_{\left((u_0,u_1),\ldots, (u_0,u_1)\right)}(0,0)
\]
for brevity, Theorem \ref{mainth} will be proved if we can construct a sequence of initial data $\left(u^N_0,u^N_1\right)_N\subset H^s\times H^{s-2}$ satisfying
\begin{equation}
\lim_{N\to \infty}\, \left\|u_0^N\right\|_{H^{s}} + \left\|u_1^N\right\|_ {H^{s-2}}\,=\,0
\label{ivhs}
\end{equation}
and 
\begin{equation}
\lim_{N\to \infty}\, \frac{\left\|A_{p}(u_0^N,u_1^N)(t)\right\|_{H^s}}{\left(\left\|u_0^N\right\|_{H^{s}} + \left\|u_1^N\right\|_ {H^{{s}-2}}\right)^p}\,=\,\infty, \qquad (\forall)\, 0 < t < T.
\label{ratio}
\end{equation}

The general profile for our initial data is
\begin{align}
\widehat{u^N_0}(\xi)\,&=\,\frac{1}{N^{\sigma}}\,\left(\varphi_{A_N}(\xi)\,+\,\varphi_{-A_N}(\xi)\right), \label{u0n}\\
\widehat{u^N_1}(\xi)\,&=\,-i \frac{1}{N^{\sigma}}\,\lambda(\xi)\,\left(\varphi_{A_N}(\xi)\,-\,\varphi_{-A_N}(\xi)\right),\label{u1n}
\end{align}
where $\sigma>s$ is a real parameter, $(A_N)_N$ is a sequence of subsets of $\R$ or $\mathbb{T}$ to be specified later, and $\varphi_A$ is the characteristic function of the set $A$. Using \eqref{L}, we obtain that
\begin{equation}\aligned
\widehat{L(u^N_0,u^N_1)}(t,\xi)\,&=\, \frac{1}{N^\sigma}\,\left(e^{-it\lambda(\xi)}\varphi_{A_N}(\xi)\,+\,e^{it\lambda(\xi)}\varphi_{-A_N}(\xi)\right)\\
&=\,\frac{1}{N^\sigma}\,e^{\mp it\lambda(\xi)}\varphi_{\pm A_N}(\xi),
\endaligned
\label{l+}
\end{equation}
where we assumed in the last expression an Einstein summation convention for $\pm$. Moreover, the sign convention is the one suggested by the above notation, i.e., if $\xi\in \pm \,A_N$, then the corresponding exponent is $\mp \,it\lambda(\xi)$. 

Subsequently, based on \eqref{dps} and \eqref{l+}, we deduce
\begin{equation}
\aligned
\widehat{A_p(u^N_0,u^N_1)}(t,\xi)\,=\,& \frac{\xi^2}{N^{p\sigma}\lambda(\xi)}\int_0^t \,\sin((t-\tau) \lambda(\xi))\,\cdot\\
\bigg[\int_{\R^{p-1}}\varphi_{\pm A_N}(\xi-\sum_{j=1}^{p-1} \eta_j)&\cdot e^{\mp i\tau\lambda(\xi-\sum_{j=1}^{p-1}\eta_j)}\,
\prod_{j=1}^{p-1}\varphi_{\pm A_N}(\eta_j)\cdot e^{\mp i\tau\lambda(\eta_j)}\,d\eta_1\ldots d\eta_{p-1}\bigg]d\tau
\endaligned
\label{apn}
\end{equation}
for the non-periodic problem and 
\begin{equation}
\aligned
\widehat{A_p(u^N_0,u^N_1)}(t,\xi)\,=\,& \frac{\xi^2}{N^{p\sigma}\lambda(\xi)}\int_0^t \,\sin((t-\tau) \lambda(\xi))\,\cdot\\
\bigg[\sum_{(\eta_1, \ldots, \eta_{p-1})\in \mathbb{Z}^{p-1}}\varphi_{\pm A_N}(\xi-\sum_{j=1}^{p-1} \eta_j)&\cdot e^{\mp i\tau\lambda(\xi-\sum_{j=1}^{p-1}\eta_j)}\,
\prod_{j=1}^{p-1}\varphi_{\pm A_N}(\eta_j)\cdot e^{\mp i\tau\lambda(\eta_j)}\bigg]d\tau
\endaligned
\label{app}
\end{equation}
for the periodic one. For both formulae, the time integral is of the type
\begin{equation}
\aligned
\int_0^t\,\sin(\alpha(t-\tau))\,e^{i\beta \tau}\,d\tau\,&=\,\frac{-\alpha}{\beta^2-\alpha^2}(\cos (\beta t)-\cos (\alpha t))\\
&+\,i\left[\frac{-\alpha}{\beta^2-\alpha^2}\sin (\beta t) + \frac{\beta}{\beta^2-\alpha^2}\sin (\alpha t)\right],
\endaligned
\label{calc}
\end{equation}
where $\alpha\neq \beta$ are real parameters. This allows us to explain better the main ideas in our argument. 

First, we localize the initial data $(u^N_0,u^N_1)$ at frequency $N$, i.e.,
\[
\eta\in \pm \,A_N\quad \Longrightarrow \quad |\eta| \approx N,\]
and we measure the output of $A_p$ only at frequency $\xi \approx 1$. Then, we can argue that
\begin{equation}
\aligned
\left\|A_{p}\left(u_0^N,u_1^N\right)(t)\right\|_{H^s}\,&\geq\, \left\|A_{p}\left(u_0^N,u_1^N\right)(t)\right\|_{H^{s}(\xi \approx 1)}\\
&\gtrsim\, \left\|A_{p}\left(u_0^N,u_1^N\right)(t)\right\|_{L^2(\xi \approx 1)}.
\endaligned
\label{aphs}
\end{equation}

Secondly, the generic term in either \eqref{apn} or \eqref{app} has the profile of the integrand in \eqref{calc} with
\begin{equation}
\alpha=\lambda(\xi) \quad \text{and} \quad \beta= -\,\epsilon_1\lambda(a_1)-\epsilon_2\lambda(a_2)-\ldots-\epsilon_p\lambda(a_p),
\label{ab}
\end{equation}
where
\begin{equation}
\xi=  \epsilon_1 a_1+\epsilon_2 a_2+ \ldots+\epsilon_p a_p, \qquad \epsilon_j=\pm \,1, \ a_j\in A_N, \ (\forall)1\leq j\leq p.\label{xi}
\end{equation}

Finally, the key point in the argument is the construction of a sequence of subsets $(A_N)_N$ such that, for all the terms in \eqref{apn}-\eqref{app} and for $N$ sufficiently large,
\begin{equation}
 |\beta|\,=\, \left\{
\begin{array}{l}
\ \beta \approx N^{2}, \quad \text{for} \quad p \ \text{odd},\\
\\
-\beta \approx N, \quad \text{for} \quad p \ \text{even}.\\
\end{array}\right.
\end{equation}
This rules out possible cancellations and reduces the analysis of the problem to the one of the generic term.

\section{Proof of Theorem~\ref{mainth} for $M=\R$}

We split the discussion into two cases, depending on the parity of $p$.

\subsection{Argument for $p$ even}
Here, the choice for the sequence $(A_N)_N$ is 
\[
A_N\,=\,[N,N+1], \quad (\forall) N \geq 1.
\]
Then, for $\xi \in \left[\frac{1}{4},\frac{1}{2}\right]$ and $N$ sufficiently large, precisely half of the terms in the representation \eqref{xi} have the coefficient $\epsilon =1$. Otherwise,
\[
\left|\epsilon_1 a_1+\epsilon_2 a_2+ \ldots+\epsilon_p a_p\right|\,\gtrsim\,N.\]
Therefore, eventually renotating the indices, we can assume that
\[
\xi=a_1+ \ldots+ a_{\frac p2} - a_{\frac{p+1}{2}}-\ldots-a_p 
\]
and 
\[
\beta=-\lambda(a_1)- \ldots- \lambda(a_{\frac p2}) + \lambda(a_{\frac{p+1}{2}})+\ldots+ \lambda(a_p). 
\]

Relying on 
\[
a^2\,\leq\,\lambda (a)\,\leq\, a^2 +\frac 12,
\]
we infer that
\begin{equation}
-\beta\,\geq\, a^2_1+ \ldots+ a^2_{\frac p2} - a^2_{\frac{p+1}{2}}-\ldots-a^2_p\,-\,\frac p2\,=\,\sum_{j=1}^{\frac p2} \left(a^2_j-a^2_{\frac{p}{2}+j}\right)\,-\,\frac p2. 
\label{-b}
\end{equation}
However, for $a,b \in A_N$, we have the estimate
\[
a^2-b^2\,\geq\,\min\left\{ 2N(a-b),(2N+2)(a-b)\right\},
\]
which implies, based on \eqref{-b},
\begin{equation}
\aligned
-\beta\,&\geq\, 2N \sum_{j=1}^{\frac p2} \left(a_j-a_{\frac{p}{2}+j}\right)\,+\, 2 \sum_{\left\{j: a_j< a_{\frac{p}{2}+j}\right\}} \left(a_j-a_{\frac{p}{2}+j}\right)\,-\,\frac p2\\
&=\,2N\xi\,+\, 2 \sum_{\left\{j: a_j< a_{\frac{p}{2}+j}\right\}} \left(a_j-a_{\frac{p}{2}+j}\right)\,-\,\frac p2\,>\,\frac{N-3p}{2}.\endaligned
\label{-bN}
\end{equation}

On the other hand, the elementary inequality
\[
\left|\lambda(a)-\lambda(b)\right|\,\lesssim \,N, \qquad (\forall) a,b \in A_N, 
\]
implies $|\beta|\lesssim N$. Hence, using \eqref{-bN}, we deduce that, for all the terms in \eqref{apn} corresponding to $\xi \in \left[\frac{1}{4},\frac{1}{2}\right]$ and for $N$ sufficiently large, the following uniform bound holds:
\begin{equation}
-\beta= |\beta| \approx N.
\label{beven}
\end{equation}

Coupling this estimate with \eqref{apn}, \eqref{calc}, and the fact that the measure of $A_N$ is 1, we obtain
\[
\left|\widehat{A_p(u^N_0,u^N_1)}(t,\xi)\right|\,\gtrsim\,\frac{1}{N^{p\sigma+1}}\,|\sin(\lambda(\xi)t)|,
\]  
which leads, based on \eqref{aphs}, to
\begin{equation}
\left\|A_{p}\left(u_0^N,u_1^N\right)(t)\right\|_{H^s}\,\gtrsim\,  \frac{1}{N^{p\sigma+1}}\, \left(\int_\frac{1}{4}^\frac{1}{2}\,\sin^2(\lambda(\xi)t)\,d\xi\right)^\frac{1}{2}.
\label{aphs2}
\end{equation} 

Using $A_N\,=\,[N,N+1]$ in \eqref{u0n} and \eqref{u1n}, direct computations infer that 
\begin{equation}
\left\|u_0^N\right\|_{H^{s}} + \left\|u_1^N\right\|_ {H^{{s}-2}}\,\approx\,N^{s-\sigma},\quad (\forall)\,N\geq 1.\label{idhs} 
\end{equation}
Thus \eqref{ivhs} follows as $\sigma>s$. 

Finally,  combining \eqref{aphs2} and \eqref{idhs}, we deduce for $N$ sufficiently large that
\begin{equation}
 \frac{\left\|A_{p}(u_0^N,u_1^N)(t)\right\|_{H^s}}{\left(\left\|u_0^N\right\|_{H^{s}} + \left\|u_1^N\right\|_ {H^{{s}-2}}\right)^p}\,\gtrsim\, \frac{1}{N^{ps+1}},\end{equation}
which yields \eqref{ratio} due to \eqref{sp}.

\subsection{Argument for $p$ odd}
The first remark we want to make here is that the previous choice for $A_N$ (i.e., $A_N\,=\,[N,N+1]$) doesn't work in this case, because, as $p$ is odd, 
\[
\left|\epsilon_1 a_1+\epsilon_2 a_2+ \ldots+\epsilon_p a_p\right|\,\approx\,N, 
\]
for all possible representations with $\epsilon_j=\pm \,1$, $a_j\in A_N$,  and $1\leq j\leq p$. 

Instead, we pick
\[
A_N\,=\,\tilde{A}_N\,\cup\,\tilde{A}_{2N}\,=\,\left[N+\frac{3(p-1)}{2p^2},N+\frac{3(p+2)}{2p^2}\right]\,\cup\, \left[2N, 2N+\frac{3}{p^2}\right], \quad (\forall) N \geq 1,
\]
the idea being that we can create something comparable to $1$ with $\frac p3$ triplets $(a,b,c)$, where $a$, $b \in \tilde{A}_N$ and $c \in \tilde{A}_{2N}$. Intuitively, in the representation \eqref{xi}, we will have $\epsilon=1$ associated to roughly $\frac{2p}{3}$ terms from $\tilde{A}_N$ and  $\epsilon=-1$ connected to $\frac p3$ terms from $\tilde{A}_{2N}$. This line of reasoning prompts a $\text{mod}\,3$ discussion in terms of $p$ as follows.

The interval on which we restrict $\xi$ to live in is given by
\begin{equation}
\xi \in I_p\,=\, \left\{
\begin{array}{l}
\left[1-\frac{2}{p}, \,1+\frac{2}{p}\right],  \qquad \qquad \qquad \ \ p\equiv 0\, (\text{mod}\,3),\\
\\
\left[1-\frac{3}{p}-\frac{4}{p^2}, \, 1-\frac{2}{p}-\frac{8}{p^2}\right], \qquad  p\equiv 1\,(\text{mod}\,3),\\
\\
\left[1-\frac{4}{p}+\frac{4}{p^2}, \,1-\frac{4}{p^2}\right], \qquad  \qquad p\equiv 2\,(\text{mod}\,3).\\
\end{array}\right.
\end{equation}

Our goal is to show that for all the terms in \eqref{apn} we have the uniform bound
\begin{equation}
\beta= |\beta| \approx N^2.
\label{bodd}
\end{equation} 
We achieve this by proving two claims:
\begin{itemize}
\item for all $\xi \in I_p$, there exists a representation of type \eqref{xi};

\item if $\xi \in I_p$ is given by a representation of type \eqref{xi}, then this expression has one the following generic profiles:
\begin{equation}
\xi=(N+N-2N)+\ldots+(N+N-2N) \ \text{if}\ p\equiv 0\, (\text{mod}\,3),
\label{0mod3}
\end{equation}
\begin{equation}
\xi= (N+N-2N)+\ldots+(N+N-2N)+(N-N)+(N-N), \ \text{if}\ p\equiv 1\, (\text{mod}\,3),\label{1mod31}
\end{equation}
\begin{equation}
\xi= (N+N-2N)+\ldots+(N+N-2N)+(N-N)+(2N-2N), \ \text{if}\ p\equiv 1\, (\text{mod}\,3), \label{1mod32}
\end{equation}
\begin{equation}
\xi= (N+N-2N)+\ldots+(N+N-2N)+(2N-2N)+(2N-2N), \ \text{if}\ p\equiv 1\,(\text{mod}\,3), \label{1mod33}
\end{equation}
\begin{equation}
\xi= (N+N-2N)+\ldots+(N+N-2N)+(N-N), \ \text{if}\ p\equiv 2\, (\text{mod}\,3), \label{2mod31}
\end{equation}
\begin{equation}
\xi= (N+N-2N)+\ldots+(N+N-2N)+(2N-2N), \ \text{if}\ p\equiv 2\, (\text{mod}\,3), \label{2mod32}
\end{equation}
\end{itemize}
where $N$ denotes an entry from $\tilde{A}_N$, $2N$ stands for one from $\tilde{A}_{2N}$, and there are precisely $2 \lfloor\frac{p-3}{6}\rfloor +1$ parentheses containing the combination $N+N-2N$. Here, $\lfloor z \rfloor$ is the largest integer not greater than $z$. 

If we assume the two claims, we obtain that 
\begin{equation} 
-\beta=(\lambda(N)+\lambda(N)-\lambda(2N))+\ldots+(\lambda(N)+\lambda(N)-\lambda(2N)) + \ \text{Error},\end{equation}
where by Error we designate any term which is not part of a triplet. Based on \eqref{0mod3}-\eqref{2mod32}, we have
\begin{equation}
\left|\text{Error}\right|\,\lesssim\, \left|\lambda(N)-\lambda(N)\right| +  \left|\lambda(2N)-\lambda(2N)\right|,\end{equation}
the notation convention regarding $N$ and $2N$ being as above. Simple computations yield, for sufficiently large $N$, 
\begin{equation}
\lambda(N)+\lambda(N)-\lambda(2N)\,=\,-2N^2 + O(N),
\end{equation} 
where $O$ is the usual big O notation, and
\begin{equation}
\left|\lambda(N)-\lambda(N)\right| +  \left|\lambda(2N)-\lambda(2N)\right|\,\lesssim\, N,
\end{equation}
both of which are uniformly in $N$.  \eqref{bodd} is then immediate.

Afterwards, the argument for verifying \eqref{ivhs} and \eqref{ratio} is almost identical to the one for $p$ even, the only true modification being generated by \eqref{bodd}, which leads to 
\begin{equation}
\left\|A_{p}\left(u_0^N,u_1^N\right)(t)\right\|_{H^s}\,\gtrsim\,  \frac{1}{N^{p\sigma+2}}\, \left(\int_{I_p}\,\sin^2(\lambda(\xi)t)\,d\xi\right)^\frac{1}{2},
\label{aphs3}
\end{equation} 
instead of \eqref{aphs2}. Other modifications involve the definitions for $A_N$ and for the interval in which $\xi$ varies, but one can see right away that these are easily manageable.  

Therefore, in order to finish the proof for this case, all we need is to verify the two claims. For the first one, it  is easy to see that the range of values for $a+b-c$, where  $a$, $b \in \tilde{A}_N$ and $c \in \tilde{A}_{2N}$, is 
\[
\left[ \frac{3(p-2)}{p^2},  \frac{3(p+2)}{p^2}
\right] .
\]
Hence, the sum of $2 \lfloor\frac{p-3}{6}\rfloor +1$ such terms generates everything in
\[
\left[ \left(2 \lfloor\frac{p-3}{6}\rfloor +1\right)\frac{3(p-2)}{p^2},  \left(2 \lfloor\frac{p-3}{6}\rfloor +1\right)\frac{3(p+2)}{p^2}
\right],
\]
which, in turn, it is straightforward to check through a $\text{mod}\,3$ analysis that it contains $I_p$. The remaining $p-6 \lfloor\frac{p-3}{6}\rfloor -3=0$, $2$, or $4$ terms in the expression of $\xi$ can then be easily filled in with $(N-N)$ or $(N-N)+(N-N)$ combinations.  

The verification of the second claim is the most involved part of the argument. We start by introducing, for each representation of $\xi\in I_p$ as \eqref{xi}, the following notation:
\begin{align}
n_1\,=\,\left|\left\{ i| a_i\in  \tilde{A}_N \ \text{and} \ \epsilon_i = 1\right\}\right|,\label{n1}\\
n_2\,=\,\left|\left\{ i| a_i\in  \tilde{A}_N \ \text{and} \ \epsilon_i = -1\right\}\right|,\label{n2}\\
n_3\,=\,\left|\left\{ i| a_i\in  \tilde{A}_{2N} \ \text{and} \ \epsilon_i = 1\right\}\right|,\label{n3}\\
n_4\,=\,\left|\left\{ i| a_i\in  \tilde{A}_{2N} \ \text{and} \ \epsilon_i = -1\right\}\right|. \label{n4}
\end{align} 

It is clear that 
\begin{equation}
n_1\,+\,n_2\,+\,n_3\,+\,n_4\,=\,p
\label{sumn}
\end{equation}
and, for $N$ sufficiently large, we also have
\begin{equation}
n_1\,-\,n_2\,+\,2(n_3\,-\,n_4)\,=\,0,
\label{difn}
\end{equation}
as $\xi\in I_p$ forces the $N$'s to cancel out. Next, we prove:

\begin{lemma}
If $N$ is sufficiently large, then any representation of $\xi\in I_p$ as \eqref{xi} has
\begin{equation}
n_1\,>\,n_2  \ \text{and} \ n_3\,<\,n_4.
\end{equation}
\end{lemma}

\begin{proof}
If $n_1=n_2$, then, by \eqref{difn}, $n_3=n_4$, which, according to \eqref{sumn}, leads to $p$ even, a contradiction. If $n_1<n_2$, then, again by \eqref{difn}, $n_3>n_4$. Using the notation $n_3-n_4=\gamma\geq 1$, we deduce, from \eqref{sumn} and \eqref{difn}, that
\begin{equation}
n_1\,=\,\frac{p-3\gamma}{2} - n_4, \qquad n_2\,=\,\frac{p+\gamma}{2} - n_4.
\label{n1n2}
\end{equation}
We can infer then
\[
\aligned
\xi\,&=\,  \epsilon_1 a_1+\epsilon_2 a_2+ \ldots+\epsilon_p a_p\\ 
&\leq\, n_1\left(N+ \frac{3(p+2)}{2p^2}\right) - n_2 \left(N+\frac{3(p-1)}{2p^2}\right) + n_3\left(2N +\frac{3}{p^2}\right) - n_4\cdot 2N\\
&=\,\frac{9}{4p} - \gamma \cdot \frac{12p+3}{4p^2} - n_4\cdot \frac{3}{2p^2}\,\leq\,\frac{9}{4p} - \frac{12p+3}{4p^2}\,<\,0,\endaligned
\]
contradicting $\xi>0$, as $\xi\in I_p$. This concludes the proof.
\end{proof}

Using this result, we prove the following bound concerning $n_2$ and $n_3$.

\begin{lemma}
If $N$ is sufficiently large, then any representation of $\xi\in I_p$ as \eqref{xi} has
\begin{equation}
n_2\,+\,n_3\,\leq\, 2.
\label{n2n3}
\end{equation}
\end{lemma} 

\begin{proof}
Based on the previous lemma, we can write $n_4-n_3=\gamma\geq 1$, which leads, as above, to
\begin{equation}
n_1\,=\,\frac{p+\gamma}{2} - n_3, \qquad n_2\,=\,\frac{p-3\gamma}{2} - n_3.
\label{n1n2v2}
\end{equation}
It follows that
\[
\aligned
\xi\,&=\,  \epsilon_1 a_1+\epsilon_2 a_2+ \ldots+\epsilon_p a_p\\ 
&\leq\, n_1\left(N+ \frac{3(p+2)}{2p^2}\right) - n_2 \left(N+\frac{3(p-1)}{2p^2}\right) + n_3\left(2N +\frac{3}{p^2}\right) - n_4\cdot 2N\\
&=\,\frac{9}{4p} + \gamma \cdot \frac{12p-3}{4p^2} - n_3\cdot \frac{3}{2p^2}\,\leq\,\frac{9}{4p} + \gamma \cdot \frac{12p-3}{4p^2}.\endaligned
\]
From \eqref{n1n2v2}, we obtain
\[
n_2+n_3\,=\,\frac{p-3\gamma}{2}.\]
Therefore, if we argue by contradiction and assume $n_2\,+\,n_3\,\geq\, 3$, we deduce
\[
\gamma \,\leq\, \frac{p-6}{3},
\]
which implies on one hand $p\geq 9$. On the other hand, used in the above estimate involving $\xi$, it leads to 
\[
\xi\,\leq\, \frac{9}{4p} + \gamma \cdot \frac{12p-3}{4p^2}\,\leq\,\frac{9}{4p} +  \frac{(p-6)(4p-1)}{4p^2}\,=\, 1-\frac{4}{p}+\frac{6}{4p^2}.\] 
One can check case by case that, for $p\geq 9$, this upper bound for $\xi$ is less than the left-hand endpoint of $I_p$, which is obviously a contradiction. 
\end{proof}

As it turns out, the two lemmas are all that is needed to fully solve, for $(n_1, n_2, n_3, n_4)$, the diophantine system composed of \eqref{sumn} and \eqref{difn}. We argue by considering the possible values of $n_2+n_3$, which, according to \eqref{n2n3}, are $0$, $1$, and $2$.

If $n_2+n_3=0$, then $n_2=n_3=0$ and, using the notation in the previous lemma, $p=3\gamma$. This implies that $ p\equiv 0\, (\text{mod}\,3)$ and, due to \eqref{n1n2v2}, we obtain
\[
(n_1,n_2,n_3,n_4)\,=\,\left(\frac{2p}{3},\,0,\,0,\,\frac{p}{3}\right),
\]
which is exactly the profile \eqref{0mod3}.

If $n_2+n_3=1$, then $p=3\gamma+2$ and either $(n_2,n_3)=(1,0)$ or $(n_2,n_3)=(0,1)$. It follows that $ p\equiv 2\, (\text{mod}\,3)$ and, again due to \eqref{n1n2v2}, we obtain either
\[
(n_1,n_2,n_3,n_4)\,=\,\left(\frac{2p-1}{3},\,1,\,0,\,\frac{p-2}{3}\right),
\]
which is described by \eqref{2mod31}, or
\[
(n_1,n_2,n_3,n_4)\,=\,\left(\frac{2p-4}{3},\,0,\,1,\,\frac{p+1}{3}\right),
\]
which is given by \eqref{2mod32}.

The case $n_2+n_3=2$ is similar, yielding \eqref{1mod31}-\eqref{1mod33}. This concludes the argument for $p$ odd, finishing also the proof for the non-periodic case of Theorem \ref{mainth}.

\section{Proof of Theorem~\ref{mainth} for $M=\mathbb{T}$}
The proof follows the same line of ideas as the one for $M=\R$, the integral formula \eqref{apn} being replaced by \eqref{app}. Moreover, it is substantially simplified by the discrete nature of this case. This is why we mention only the relevant differences and leave the details for the interested reader.

\subsection{Argument for $p$ even} Here, we choose
\[
A_N\,=\,\left\{N,\, N+1\right\}\ \text{and} \ \xi\,=\,\frac{p}{2},
\]
the only possible representation being given by
\[
\xi\,=\,(N+1-N)\,+\,\ldots\,+\, (N+1-N).
\]
Therefore,
\[
-\beta\,=\,\frac{p}{2}\, \left(\lambda(N+1)-\lambda (N)\right),
\]
which implies the desired bound
\[
-\beta= |\beta| \approx N.
\]

\subsection{Argument for $p$ odd} In this instance, we take
\[
A_N\,=\,\left\{N,\, N+1, \, 2N\right\}\ \text{and} \ \xi\,=\,\frac{p+1}{2}\,+\, \lfloor\frac{p-3}{6}\rfloor.
\]
If one uses $\tilde{A}_N$ for $\{N,\,N+1\}$ and $\tilde{A}_{2N}$ for $\{2N\}$ in the context of the notations \eqref{n1}-\eqref{n4}, it can be proved, as before, that
\[
n_1\,>\,n_2, \quad n_3\,<\,n_4, \quad \text{and} \quad n_2+n_3\, \leq\, 2.\]
This allows us to solve the diophantine system, but, due to the specific value of $\xi$, a reduced number of generic profiles are obtained by comparison with the non-periodic case:
\begin{equation}
\xi=(N+1+N+1-2N)+\ldots+(N+1+N+1-2N) \ \text{if}\ p\equiv 0\, (\text{mod}\,3),
\label{0mod3p}
\end{equation}
\begin{equation}
\aligned
\xi= (N+1+N&+1-2N)+\ldots+(N+1+N+1-2N)\\
&+(N+1-N)+(N+1-N) \ \text{if}\ p\equiv 1\, (\text{mod}\,3),
\endaligned
\label{1mod3p}
\end{equation}
\begin{equation}
\aligned
\xi= (N+1+N+1-2N)&+\ldots+(N+1+N+1-2N)\\
&+(N+1-N) \ \text{if}\ p\equiv 2\, (\text{mod}\,3), 
\endaligned
\label{2mod3p}
\end{equation}
For all such representations, it is straightforward to deduce
\[
\beta= |\beta| \approx N^2,
\]
which finishes the proof.

\section*{Acknowledgements}
We are indebted to Nobu Kishimoto for pointing out to us the insufficiency of \eqref{ivhs} and \eqref{ratio} in claiming the discontinuity of the flow map. The first author was supported in part by the National Science Foundation Career grant DMS-0747656.

\bibliographystyle{amsplain}
\bibliography{bousbib}

\end{document}